\documentclass[12pt,leqno]{amsart}
\usepackage{amssymb}
\usepackage{amsthm}
\usepackage{tikz}

\newtheorem{theorem}{Theorem}
\newtheorem{remark}{Remark}
\newtheorem{question}{Question}
\newtheorem{lemma}{Lemma}
\newtheorem{proposition}{Proposition}
\newtheorem{corollary}{Corollary}

\def\bF{\mathbb F}

\def\bC{\mathbb C}

\DeclareMathOperator{\Pic}{Pic}

\begin{document}
\title{ Arithmetic of  "units" in $\mathbb F_q[T]$ }

\author{Bruno Angl\`es \and Mohamed Ould Douh}

\address{
Universit\'e de Caen, 
CNRS UMR 6139, 
Campus II, Boulevard Mar\'echal Juin, 
B.P. 5186, 
14032 Caen Cedex, France.
}
\email{bruno.angles@unicaen.fr, mohamed.douh@unicaen.fr}

\maketitle

\noindent The aim of this note is to study the arithmetic of  Taelman's unit module for  $A:= \mathbb F_q[T].$  This module is the $A$-module (via the Carlitz module) generated by $1.$ Let $P$ be a monic irreducible polynomial in $A,$ we show that the "$P$-adic behaviour" of $1$ is connected to some isotypic component of the ideal class group of the integral closure of $A$ in the $P$th cyclotomic function field. The results contained in this note are applications of the deep results obtained by L. Taelman in \cite{TAE3}.\par
\noindent The authors thank L. Taelman and D. Thakur for fruitfull discussions.\par
\section{Background on the Carlitz module}

\noindent  Let $\mathbb F_q$ be a finite field having $q$ elements, $q\geq 3,$  and let $p$ be the characteristic of $\mathbb F_q.$ Let $T$ be an indeterminate over $\mathbb F_q,$ and set : $ k:= \mathbb F_q(T),$ $A:= \mathbb F_q[T],$ $A_{+}:=\{ a\in A, \, a \, {\rm monic}\, \}.$ A prime in $A$ will be a monic irreducible polynomial in $A.$ Let $\infty$ be the unique place of $k$ which is a pole of $T,$ and set : $k_{\infty }:= \mathbb F_q((\frac{1}{T})).$ Let $\bC_{\infty}$ be a completion of an algebraic closure of $k_{\infty},$ then $\bC_{\infty}$ is algebraically closed and complete and we denote by $v_{\infty}$ the valuation on $\bC_{\infty}$ normalized such that $v_{\infty} (T)=-1.$ We fix an embedding of an algebraic closure of $k$ in $\bC_{\infty},$ and thus all the finite extensions of $k$  considered in this note will be contained in $\bC_{\infty}.$ \par
\noindent Let $L/k$ be a finite extension, we denote by :\par
\noindent - $S_{\infty}(L):$ the set of places of $L$ above $\infty ,$ if $w\in S_{\infty}(L)$ we denote the completion of $L$ at $w$ by $L_w$ and we view $L_w$ as a subfield of $\bC_{\infty},$\par
\noindent - $O_L:$ the integral closure of $A$ in $L,$\par
\noindent - $\Pic (O_L):$ the ideal class group of $L,$\par
\noindent - $L_{\infty}:$ the $k_{\infty}$-algebra $L\otimes_kk_{\infty},$ recall that we have a natural isomorphism of $k_{\infty}$-algebras : $L_{\infty}\simeq \prod_{w\in S_{\infty}(L)} L_w.$\par

\subsection{The Carlitz exponential} ${}$\par
\noindent Set $D_0=1$ and for $i\geq 1, \, D_i=(T^{q^i}-T)D_{i-1}^q.$ The Carlitz exponential is defined by :
$$e_C(X)=\sum_{i\geq 0}\frac{X^{q^i}}{D_i}\, \in k[[X]].$$
Since $\forall i\geq 0,$ $v_{\infty}(D_i) =-iq^i,$ we deduce that $e_C$ defines an entire function on $\bC_{\infty}$ and that $e_C(\bC_{\infty})=\bC_{\infty}.$ Observe that :
$$e_C(TX) = Te_C(X)+ e_C(X)^q.$$
Thus, $\forall a\in A,$ there exists a  $\bF_q$-linear polynomial $\phi_a(X)\in A[X]$ such that $e_C(aX)=\phi_a(e_C(X)).$ The map  $\phi : A\rightarrow End_{\bF_q}(A),$ $a\mapsto \phi_a,$ is an injective morphism of $\bF_q$-algebras called the Carlitz module.\par
\noindent Let $\varepsilon _C ={}^{q-1}\sqrt{T-T^q}\prod_{j\geq 1}(1-\frac{T^{q^j }-T}{T^{q^{j+1}}-T})\, \in \mathbb C_{\infty}.$ Then by \cite{GOS} Theorem 3.2.8,  we have the following equality in $\bC_{\infty}[[X]]:$
$$e_C(X)=X\prod_{\alpha \in \varepsilon_CA\setminus \{0\}}(1-\frac{X}{\alpha}).$$
Note that $v_{\infty}(\varepsilon_C)=-\frac{q}{q-1}.$ 
Let $log_C(X)\in k[[X]]$ be the formal inverse of $e_C(X),$ i.e. $e_C(log_C(X))=log_C(e_C(X))=X.$ Then by \cite{GOS} page 57,  we have :
$$log_C(X)=\sum_{i\geq 0}\frac{X^{q^i}}{L_i},$$
where $L_0=1,$ and for $i\geq 1, $ $L_i =(T-T^{q^i}) L_{i-1}.$ Observe that $\forall i\geq 0, \, v_{\infty}(L_i)= -\frac{q^{i+1}-q}{q-1}.$ Therefore $log_C$ converges on $\{ \alpha \in \bC_{\infty}, \, v_{\infty}(\alpha )> -\frac{q}{q-1}\}.$ Furthermore, for $\alpha$ in $\bC_{\infty }$ such that $v_{\infty }(\alpha)> -\frac{q}{q-1},$ we have :\par
\noindent - $v_{\infty} (e_C(\alpha))=v_{\infty}(log_C(\alpha))=v_{\infty}(\alpha ),$\par
\noindent - $e_C(log_C(\alpha))=log_C(e_C(\alpha))=\alpha.$\par

\subsection{Torsion points}${}$\par
\noindent We recall some basic properties of cyclotomic function fields. For a nice introduction to the arithmetic properties of such fields, we refer the reader to \cite{ROS} chapter 12. Let $P$ be a prime of $A$ of degree $d.$ Set $\Lambda_P:=\{ \alpha\in \bC_{\infty}, \, \phi_P(\alpha)=0\}.$ Note that the elements of $\Lambda_P$ are integral over $A,$ and that $\Lambda_P$ is a $A$-module via $\phi$ which is isomorphic to $\frac{A}{PA}.$ Set $\lambda_P=e_C(\frac{\varepsilon_C}{P}),$ then $\lambda_P$ is a generator of the $A$-module $\Lambda_P.$ Let $K_P=k(\Lambda_P)=k(\lambda_P).$ We have the following properties  :\par
\noindent - $K_P/k$ is an abelian extension of degree $q^d-1,$\par
\noindent - $K_P/k$ is unramified outside $P, \infty,$\par
\noindent - let $R_P=O_{K_P},$ then $R_P=A[\lambda_P],$\par
\noindent - if $w\in S_{\infty}(K_P),$ the completion of $K_P$ at $w$ is equal to $k_{\infty}(\varepsilon _C),$  in particular the decomposition group at $w$ is equal to the inertia group at $w$ and is isomorphic to $\bF_q^*,$ furthermore  $\mid S_{\infty}(K_P)\mid =\frac{q^d-1}{q-1},$\par
\noindent - $K_P/k$ is totally ramified at $P$ and the unique prime ideal of $R_P$ above $P$ is equal to $\lambda_PR_P,$\par
\noindent Let $\Delta={\rm Gal}(K_P/k).$ For $a\in A\setminus PA,$ we denote by $\sigma _a$ the element in $\Delta$ such that $\sigma _a(\lambda_P)=\phi_a(\lambda_P).$ The map : $A\setminus PA \rightarrow \Delta, $ $a\mapsto \sigma_a$ induces an isomorphism of groups :
$$(\frac{A}{PA})^*\simeq \Delta.$$

\subsection{The unit module and the class module}${}$\par
\noindent Let $R$ be an $A$-algebra, we denote by $C(R)$ the $\bF_q$-algebra $R$ equipped with the $A$-module structure induced by $\phi,$ i.e. : $\forall r\in C(R),$ $T.r= \phi_T(r)= Tr+r^q.$ For example, the Carlitz exponential induces the following exact sequence of $A$-modules :
$$0\rightarrow \varepsilon_CA\rightarrow \bC_{\infty}\rightarrow C(\bC_{\infty})\rightarrow 0.$$
Let $L/K$ be a finite extension, then B. Poonen has proved in \cite{POO} that $C(O_L)$ is not a finitely generated $A$-module. Recently, L. Taelman has introduced  in \cite{TAE1} a natural sub-$A$-module of $C(O_L)$ which is finitely generated and called the unit module associated to $L$ and $\phi.$ Fisrt note that the Carlitz exponential induces a morphism of $A$-modules : $L_{\infty} \rightarrow C(L_{\infty}),$ and the kernel of this map is a free $A$-module of rank $\mid\{ w\in S_{\infty}(L),\, \varepsilon_C \in L_w\} \mid.$ Now, let's consider the  natural map of $A$-modules  induced by the inclusion $C(O_L)\subset C(L_{\infty}):$
$$\alpha_L :\, C(O_L) \rightarrow \frac{C(L_{\infty})}{e_C(L_{\infty})}.$$
L. Taelman has proved the following remarkable results (\cite{TAE1}, Theorem 1, Corollary 1):\par
\noindent - $U(O_L):= {\rm Ker}\alpha_L$ is a finitely generated $A$-module of rank $[L:k]-\mid\{ w\in S_{\infty}(L),\, \varepsilon_C \in L_w\} \mid,$ the $A$-module (via $\phi$) $U(O_L)$ is called the unit module attached to $L$ and $\phi,$\par
\noindent - $H(O_L):= {\rm Coker} \alpha_L$ is a finite $A$-module called the class module associated to $L$ and $\phi .$\par
\noindent Set:
$$\zeta_{O_L}(1):= \sum_{I\not = (0)}\frac{1}{[\frac{O_L}{I}]_A}\, \in k_{\infty},$$
where the sum is taken over the non-zero ideals of $O_L,$ and where for any finite $A$-module $M,$ $[M]_A$ denotes the monic generator of the Fitting ideal of the finite $A$-module $M.$ Then, we have the following class number formula (\cite{TAE2}, Theorem 1) :
$$\zeta_{O_L}(1) =[H(O_L)]_A\, [O_L:e_C^{-1}(O_L)]\,,$$ 
where $[O_L:e_C^{-1}(O_L)]\in k_{\infty}^*$ is a kind of regulator (see \cite{TAE2} for more details).\par

\section {The unit module for $\mathbb  F_q[T]$}

\subsection{Sums of polynomials}${}$\par
\noindent In this paragraph, we recall some computations made by G. Anderson and D. Thakur (\cite{AND&THA} pages 183,184).\par
\noindent Let $X,Y$ be two indeterminates over $k.$ We define the polynomial $\Psi_k(X)\in A[X]$ by the following identity :
$$e_C(Xlog_C(Y))=\sum_{k\geq 0}\Psi_k(X)Y^{q^k}.$$
We have that $\Psi_0(X)=X$ and for $k\geq 1:$
$$\Psi_k(X)=\sum_{i=0}^{k}\frac{1}{D_i(L_{k-i})^{q^i}}X^{q^i}.$$
For $a=a_0+a_1 T+\cdots +a_nT^n , $ $a_0,\cdots a_n \in \mathbb F_q,$ we have :
$$\phi_a(X)=\sum_{i=0}^n [^{a}_i] X^{q^i},$$
where $[^{a}_i] \in A$ for $i=0,\cdots ,n,$ $[^{a}_0]=a$ and $[^{a}_n]=a_n.$ But since $e_C(aX)=\phi_a(e_C(X)),$ we deduce that for $k\geq 1$ :
$$\Psi_k (X) =\frac{1}{D_k} \prod_{a\in A(d)}(X-a),$$
where $A(d)$ is the set of elements in $A$ of degree strictly less than $k.$ In particular :
$$\Psi_k(X+T^k) =\Psi_k(X)+1 =\frac{1}{D_k}\prod_{a\in A_{+,k}}(X+a),$$
where $A_{+,k}$ is the set of monic elements in $A$ of degree $k.$ Now for $j\in \mathbb N$ and for $i \in \mathbb Z,$ set :
$$S_j(i) =\sum_{a\in A_{+,j}}a^i \, \in k.$$
Note that the derivative of $\Psi_k(X)$ is equal to $\frac{1}{L_k}.$ Therefore we get :
$$\frac{1}{L_k} \frac{1}{\Psi_k(X)+1}=\sum_{a\in A_{+,k}}\frac{1}{X+a}.$$
Thus:
$$\frac{1}{L_k} \frac{1}{\Psi_k(X)+1}=\sum_{n\geq 0} (-1)^n S_k(-n-1) X^{n}.$$
But :
$$\Psi_k(X) \equiv \frac{1}{L_k} X \mod{X^q}.$$
Therefore :
$$\forall k\geq 0, \, {\rm for}\, c\in \{ 1, \cdots , q-1\},\, S_k(-c) =\frac{1}{L_k^c} .$$
But observe that we also have :
$$\frac{1}{L_k} \frac{1}{\Psi_k(X)+1}=\sum_{n\geq 0} (-1)^n S_k(n) X^{-n-1}.$$
But :
$$\frac{1}{\Psi_k(X)+1}\equiv 0\pmod{X^{-q^k}}.$$
Therefore :
$$\forall k\geq 0, \, {\rm for}\, i\in \{0,\cdots , q^k -2\}, \, S_k(i)=0.$$
The Bernoulli-Goss numbers, $B(i)$ for $i\in \mathbb N,$  are elements of $A$ defined as follows :\par
\noindent - $B(0)=1,$\par
\noindent - if $i\geq 1$ and $i\not \equiv 0\pmod{q-1},$ $B(i) =\sum_{j\geq 0}S_j(i)$ which is a finite sum by our previous discussion,\par
\noindent - if $i\geq 1,$ $i\equiv 0\pmod{q-1},$ $B(i) =\sum_{j\geq 0} jS_j(i) \in A.$\par
\noindent We have :
\begin{lemma}
\label{lemma1}
Let $P$ be a prime of $A$ of degree $d$ and let $c\in \{2, \cdots , q-1\}.$ Then :
$$B(q^d-c)\equiv \sum_{k=0}^{d-1} \frac{1}{L_k^{c-1}}\pmod{P}.$$
\end{lemma}
\begin{proof} Note that $q^d-c$ is not divisible by $q-1$ and that $1\leq q^d-c< q^d-1.$ Thus :
$$B(q^d-c)=\sum_{k=0}^{d-1} S_k(q^d-c).$$
Now, for $k\in \{0,\cdots , d-1\},$ we have :
$$S_k(q^d-c)\equiv S_k(1-c)\pmod {P}.$$
The lemma follows by our previous computations.
\end{proof}
\noindent  We will also nee some properties of the polynomial $\Psi_k :$
\begin{lemma}
\label{lemma2}
${}$\par
\noindent 1) Let $X,Y$ be two indeterminates over $k.$ We have :
$$\forall k\geq 0,\, \Psi_k(XY)=\sum_{i=0}^{k}\Psi_i(X)\Psi_{k-i}(Y)^{q^i}.$$
2) For $k \geq 0,$ we have :
$$\psi_{k+1}(X)=\frac{\Psi_k(X)^q-\Psi_k(X)}{T^{q^{k+1}}-T}.$$
\end{lemma}
\begin{proof} ${}$\par
\noindent 1) Recall that we have seen that :
$$\forall a\in A,\, \phi_a(X)=\sum_{k\geq 0} \Psi_k(a) X^{q^k}.$$
Furthermore, for $a\in A :$
$$e_C(aXlog_C(Y))=\phi_a(e_C(Xlog_C(Y))).$$
Thus, for all $a\in A:$
$$\forall k\geq 0,\, \Psi_k(aX)=\sum_{i=0}^{k}\Psi_i(a)\Psi_{k-i}(X)^{q^i}.$$
The first assertion of the lemma follows.\par
\noindent 2) For all $a\in A,$ we have :
$$\phi_a(TX+X^q)=T\phi_a(X)+\phi_a(X)^q.$$
Thus, for all $a\in A :$
$$\forall k\geq 0, \, \psi_{k+1}(a)=\frac{\Psi_k(a)^q-\Psi_k(a)}{T^{q^{k+1}}-T}.$$
\end{proof}
\begin{lemma}
\label{lemma3}
\noindent Let $P$ be a prime of $A$ of degree $d.$ We have :
$$\phi_P(X)=\sum_{k=0}^d [^{P}_k]X^{q^k},$$
where $[^{P}_0]=P$ and $[^{P}_d]=1.$ Then, for $k=0, \cdots , d-1,$ $P$ divides $[^{P}_k]$ and :
$$\frac{[^{P}_k]}{P}\equiv \frac{1}{L_k}\pmod{P}.$$
\end{lemma}
\begin{proof} Since $[^{P}_k]=\Psi_k(P),$ the lemma follows from the second assertion of lemma \ref{lemma2}.
\end{proof}
\noindent If we combine lemma \ref{lemma1} and lemma \ref{lemma3}, we get :
\begin{corollary}
\label{corollary1}
${}$\par
\noindent Let $P$ be a prime of $A$ of degree $d.$ Then :
$$\phi_{P-1}(1)\equiv PB(q^d-2)\pmod{P^2}.$$
\end{corollary}
\begin{remark} D. Thakur has informed the authors that the  congruence in corollary  \ref{corollary1} was already observed by him   in \cite{THA}. 
\end{remark}

\subsection{The unit module for $\mathbb F_{q^n }[T]$}${}$\par
\noindent  Set $k_n=\mathbb F_{q^n}(T)$ and $A_n=\mathbb F_{q^n}[T].$ In this paragraph we will determine $U(A_n)$ and $H(A_n).$ We have :
$$k_{n,\infty}=k_n\otimes _k k_{\infty}=\mathbb F_{q^n }((\frac{1}{T})).$$
Let $\varphi$ be the Frobenius of $\mathbb F_{q^n}/\mathbb F_q,$ recall that $k_n/k$ is a cyclic extension of degree $n$ and its Galois group is generated by $\varphi.$ Set $G={\rm Gal}(k_n/k)$ and let $\alpha \in \mathbb F_{q^n}$ which generates a normal basis of $\mathbb F_{q^n}/\mathbb F_q.$ Then $A_n$ is a free $A[G]$-module of rank one generated by $\alpha .$ Note that :
$$k_{n,\infty} =A_n\oplus  \frac{1}{T} \mathbb F_{q^n}[[\frac{1}{T}]].$$
By the results of paragraph 1.1 :
$$log_C(\alpha) \in \mathbb F_{q^n}[[\frac{1}{T}]]^*,$$
and :
$$e_C(\frac{1}{T} \mathbb F_{q^n}[[\frac{1}{T}]]) = \frac{1}{T} \mathbb F_{q^n}[[\frac{1}{T}]].$$
Now :
$$k_{n,\infty}=\oplus_{i=0}^{n-1}k_{\infty}log_C(\alpha^{q^i}).$$
Thus :
$$k_{n,\infty}= \frac{1}{T} \mathbb F_{q^n}[[\frac{1}{T}]]\oplus (\oplus_{i=0}^{n-1} A\, log_C(\alpha^{q^i})).$$
Let $\mathfrak S_n(A)$ be the sub-$A$-module of $C(A_n)$ generated by $\mathbb F_{q^n},$ then $\mathfrak S_n(A)$ is a free $A$-module of rank $n$ generated by $\{ \alpha, \alpha ^q,\cdots , \alpha^{q^{n-1}} \}.$ We have  :
$$e_C(k_{n,\infty})= \mathfrak S_n(A)\oplus  \frac{1}{T} \mathbb F_{q^n}[[\frac{1}{T}]].$$
Thus :
$$U(A_n)=A_n\cap e_C(k_{n,\infty})= \mathfrak S_n (A),$$
and :
$$H(A_n)=\frac  {C(k_{n,\infty})}{ C(A_n)+e_C(k_{n,\infty})}=\{ 0\}.$$
In particular, for $n=1,$ we get $U(A)=\mathfrak S_1(A)=$ the free $A$-module of rank one generated (via $\phi$) by $1$ and $H(A)=\{ 0\}.$\par
\noindent Let $F\in k_{\infty}[G]$ be defined by :
$$F=\sum_{i= 0}^{n-1} (\sum_{j\equiv i\pmod{n}}\frac{1}{L_j} )\varphi ^i.$$
Then :
$$e_C^{-1} (A_n) =\oplus_{i=0}^{n-1} A\,  log_C(\alpha ^{q ^i}) =F A_n.$$
Write $n=mp^{\ell},$ where $\ell \geq 0$ and $m\not \equiv 0\pmod{p}.$ Let $\mu_{m}=\{ x\in \bC_{\infty} ,\, x^m=1\}$ which is a cyclic group of order $m.$ Then we can compute  Taelman's regulator (just calculate the "determinant" of $F$) :
$$[A_n:e_C^{-1} (A_n)]= ((-1)^{m-1}\prod_{\zeta\in \mu_m} (\sum_{i= 0}^{n-1} (\sum_{j\equiv i\pmod{n}}\frac{1}{L_j} )\zeta^i))^{p^{\ell}}.$$
Thus, Taelman's class number formula becomes  in this case :
$$\zeta_{A_n}(1)=((-1)^{m-1}\prod_{\zeta\in \mu_m} (\sum_{i= 0}^{n-1} (\sum_{j\equiv i\pmod{n}}\frac{1}{L_j} )\zeta^i))^{p^{\ell}}.$$
In particular, we get the following formula already known by Carlitz :
$$\zeta_A(1)=log_C(1).$$
\subsection {The $P$-adic behavior of "$1$"}${}$\par
\noindent Let $P$ be a prime of $A$ of degree $d.$ Let $\bC_P$ be a completion of an algebraic closure of the $P$-adic completion of $k.$ Let $v_P$ be the valuation on $\bC_P$ such that $v_P(P)=1.$ For $x\in \mathbb R,$ we denote the integer part of $x$ by $[x].$ Let $i\in \mathbb N\setminus \{0\}$ and observe that $v_P(T^{q^i}-T)= 1$ if $d$ divides $i$ and $v_P(T^{q^i}-T)= 0$ otherwise. Therefore :\par
\noindent - for $i\geq 0,$ $v_P(L_i)=[i/d],$\par
\noindent - for $i\geq 0,$ $v_P(D_i) = \frac{q^i - q^{i-[i/d]d}}{q^d-1}.$\par
\noindent This implies that $log_C(\alpha)$ converges for $\alpha \in \bC_P$ such that $v_P(\alpha )>0,$ and that $e_C(\alpha )$ converges for $\alpha \in \bC_P$ such that $v_P(\alpha )> \frac{1}{q^d-1}.$ Furthermore, for $\alpha \in \bC_P$ such that $v_P(\alpha)>\frac{1}{q^d-1},$ we have :\par
\noindent - $v_P(e_C(\alpha ))=v_P(log_C(\alpha ))=v_P(\alpha),$\par
\noindent - $e_C(log_C(\alpha ))=log_C(e_C (\alpha ))=\alpha .$\par
\begin{lemma}
\label{lemma4}
Let $A_P$ be the $P$-adic completion of $A.$ There exists $x\in A_P$ such that $\phi_P(x)=\phi_{P-1}(1)$ if and only if $\phi_{P-1}(1) \equiv 0\pmod{P^2}.$
\end{lemma}
\begin{proof}${}$\par
 \noindent First assume that $\phi_{P-1}(1)\not \equiv 0\pmod{P^2}.$ By lemma \ref{lemma3}, we have that  $v_P(\phi_{P-1}(1))=1,$ and therefore $\phi_P(X)-\phi_{P-1}(1) \in A_P[X]$ is an Eisenstein polynomial. In particular there $\phi_{P-1}(1) \not \in \phi_P(A_P).$\par
 \noindent Now, let's assume that $\phi_{P-1}(1)\equiv 0\pmod{P^2}.$ Then $v_P(log_C (\phi_{P-1}(1)))=v_P(\phi_{P-1}(1)).$ Therefore, there exists $y\in P A_P$ such that :
 $$log_C (\phi_{P-1}(1))=Py.$$
 Set $x=e_C(y) \in PA_P.$ We have :
 $$\phi_P(x) = e_C(Py)=e_C(log_C (\phi_{P-1}(1)))= \phi_{P-1}(1).$$
\end{proof}
\begin{remark} Since $1$ is an  Anderson's special point for the Carlitz module, the above  lemma can also be deduced by corollary \ref{corollary1} and  the work of G. Anderson in \cite{AND}. \end{remark}

\section{Hilbert class fields and the unit module for $\mathbb F_q[T]$}
${}$\par
\noindent Let $P$ be a prime of $A$ of degree $d.$ Recall that $K_P$ is the $P$th-cyclotomic function field, i.e. the finite extension of $k$ obtained by adjoining to $k$ the $P$th-torsion points of the Carlitz module. Let $R_P$ be the integral closure of $A$ in $K_P$ and let $\Delta$ be the Galois group of $K_P/k.$ Recall that $\Delta$ is a cyclic group of order $q^d-1$ (see paragraph 1.2). Recall that the unit module $U(A)$ is the free $A$-module (via $\phi$) generated by $1$ (see paragraph 2.2).\par
\subsection{Kummer theory}${}$\par
\noindent 
We will need the following lemma:\par
\begin{lemma}
\label{lemma5}
The natural morphism of $A$-modules :$\frac{U(A)}{P.U(A)}\rightarrow \frac{C(K_P)}{P.C(K_P)}$ induced by the inclusion $U(A)\subset C(K_P),$ is an injective map. 
\end{lemma}
\begin{proof} Recall that $K_{P,\infty}=K_P\otimes_kk_{\infty}.$ Let $Tr: K_{P,\infty}\rightarrow k_{\infty}$ be the trace map. Now let $x\in U(A) \cap P.C(K_P).$ 
Then there exits $z\in K_P$ such that $\phi_P(z)=x.$ Since $e_C(K_{P,\infty})$ is $A$-divisible,  we get that $z\in U(R_P).$  Thus $Tr(z) \in U(A).$ But :
$$-x=\phi_P(Tr(z)).$$
Therefore $x\in P.U(A).$
\end{proof}
\noindent  Let $\mathfrak U=\{ z\in \bC_{\infty},\, \phi_P(z)\in U(A)\}.$ Then $\mathfrak U$ is an $A$-module (via $\phi $) and $P.\mathfrak U =U(A).$ Therefore the multiplication by $P$ gives rise to the following exact sequence of $A$-module:
$$0\rightarrow \Lambda_P\oplus U(A) \rightarrow \mathfrak U \rightarrow \frac{U(A)}{P.U(A)}\rightarrow 0.$$
Set $\gamma =e_C(\frac{P-1}{P} log_C(1)).$ Then $\gamma \in \mathfrak U.$ Set $L=K_P(\mathfrak U).$ By the above exact sequence, we observe that :
$$L= K_P(\gamma ).$$
Furthermore  $L/k$ is a Galois extension and we set : $G={\rm Gal}(L/K_P)$ and $\mathfrak G ={\rm Gal}(L/k).$
Let $\delta \in \Delta$ and select $\widetilde {\delta} \in \mathfrak G$ such that the restriction of $\widetilde {\delta}$ to $K_P$ is equal to $\delta.$ Let $g\in G, $ then $\widetilde {\delta} g\widetilde {\delta}^{-1} \in G$ does not depend on the choice of $\widetilde {\delta}.$ Therefore $G$ is a $\mathbb F_p[\Delta]$ -module.
\begin{lemma}
\label{lemma6} We have a natural isomorphism of $\mathbb F_p[\Delta]$-modules :
$$G\simeq {\rm Hom}_A(\frac{U(A)}{P.U(A)}, \Lambda_P).$$

\end{lemma}
\begin{proof} ${}$\par
\noindent Recall that the multiplication by $P$ induces an $A$-isomorphism :
$$\frac{\mathfrak U}{\Lambda_P\oplus U(A)}\simeq \frac{U(A)}{P.U(A)}.$$
For $z\in \mathfrak U$ and $g\in G,$ set :
$$<z,g> = z-g(z) \in \Lambda_P.$$
One can verify that :\par
\noindent - $ \forall z_1,z_2\in \mathfrak U, $ $\forall g\in G,$ $<z_1+z_2, g>= <z_1, g>+<z_2, g>,$\par
\noindent - $\forall z\in \mathfrak U,$  $ \forall g_1,g_2 \in G,$ $ <z, g_1g_2>= <z,g_1>+<z,g_2>,$\par
\noindent - $ \forall z\in \mathfrak U, \forall a \in A, \forall g\in G,$ $<\phi_a(z),g>=\phi_a(<z,g>),$\par
\noindent - $\forall z\in \mathfrak U,\forall g\in G, \forall \delta \in \Delta,$ $<\widetilde  \delta (z), \delta.g>= \delta(<z,g>),$ where $\widetilde \delta \in \mathfrak G$ is such that its restriction to $K_P$ is equal to $\delta,$\par
\noindent - let $g\in G$ then : $<z,g>=0$  $\forall z\in \mathfrak U$ if and only if $g=1.$\par
\noindent Let $z\in \mathfrak U$ be such that $<z,g>=0$ $\forall g\in G.$ Then $z\in \mathfrak U^G.$ Thus $z\in K_P$ and $\phi_P(z)\in U(A).$ Thus, by lemma \ref{lemma5}, we get $\phi_P(z)\in P.U(A),$ and therefore $z\in \Lambda_P\oplus U(A).$ \par

\noindent We deduce from above that $<.,.>$ induces a non-degenerated and $\Delta$-equivariant  bilinear map :
$$\frac {U(A)}{P.U(A)}\times G \rightarrow \Lambda_P.$$

\end{proof}

\subsection{Class groups}${}$\par
\noindent We will need the following result wich is implicit in \cite{TAE3} :
\begin{proposition}
\label{proposition1}
Let $Cl^0(K_P)$ be the group of classes of divisors of degree zero of $K_P.$ Then there exists an $\frac{A}{PA}[\Delta]$-morphism :
$${\rm Hom}_A(H(R_P), \Lambda_P)\rightarrow Cl^0(K_P)\otimes_{\mathbb Z} \frac{A}{PA},$$
such that its kernel and cokernel are cyclic $\frac{A}{PA}[\Delta]$-modules.
\end{proposition}
\begin{proof} This is a consequence of  the following results in \cite{TAE3} :  exact sequence (2), Theorem 2  and Theorem 6.
\end{proof}
\noindent Let $\omega _P : \Delta \simeq (A/PA)^*$ be the cyclotomic character, i.e. $\forall a\in A\setminus PA,$ $\omega _P(\sigma_a)\equiv a \pmod{P}.$ Let $W=\mathbb Z_p[\mu_{q^d-1}],$ and fix $\rho : A/PA \rightarrow W/pW$ a $\mathbb F_p$-isomorphism. We still denote by $\omega_P$ the morphism of groups $\Delta \simeq \mu_{q^d-1}$ wich sends $\sigma_a $ to the unique root of unity congruent to $\rho (\omega_P(a))$ modulo $pW.$ Observe that $\widehat \Delta : = {\rm Hom} (\Delta, W^*)$ is a cyclic group of order $q^d-1$ generated by $\omega_P.$ For $\chi \in \widehat \Delta ,$ we set :\par
\noindent - $e_{\chi} =\frac{1}{q^d-1} \sum_{\delta \in \Delta} \chi (\delta) \delta^{-1} \in W[\Delta],$\par
\noindent - $ [\chi] =\{ \chi ^{p^j}, \, j\geq 0\} \subset \widehat \Delta,$\par
\noindent  - $e_{[\chi]}=\sum_{\psi \in [\chi ]} e_{\psi } \in \mathbb Z_p[\Delta].$\par
\noindent Let ${\rm Pic}(R_P)$ be the ideal class group of the Dedekind domain $R_P.$
\begin{corollary}
\label{corollary2}
The $\mathbb Z_p[\Delta]$-module : $e_{[\omega_P]}({\rm Pic}(R_P)\otimes _\mathbb Z \mathbb Z_p)$ is a cyclic module. Furthermore, it is non trivial if and only if $B(q^d-2)\equiv 0\pmod{P}.$
\end{corollary}

\begin{proof}${}$\par
\noindent  Recall that $H(A)=\{ 0\}.$ Note that the trace map induces a surjective morphism of $A$-modules $H(R_P)\rightarrow H(A).$ Therefore :
$$H(R_P)^{\Delta}=\{0\}.$$
Now, note that, $\forall \chi \in \widehat \Delta ,$ we have an isomorphism of $W$-modules :
$$e_{\chi } (Cl^0(K_P)\otimes_{\mathbb Z}W)\simeq e_{\chi^p } (Cl^0(K_P)\otimes_{\mathbb Z}W).$$
Thus, by proposition \ref{proposition1}, we get that $e_{[\omega_P]}(Cl^0(K_P)\otimes _\mathbb Z \mathbb Z_p)$ is a cyclic $\mathbb Z_p[\Delta]$-module. Furthermore, by \cite{GOS&SIN},  this latter module is non-trivial if and only if $B(q^d-2)\equiv 0\pmod{P}.$ We conclude the proof by noting that :
$$e_{[\omega_P]}(Cl^0(K_P)\otimes _\mathbb Z \mathbb Z_p)\simeq e_{[\omega_P]}({\rm Pic}(R_P)\otimes _\mathbb Z \mathbb Z_p).$$
\end{proof}
\noindent  Recall that $L=K_P(\gamma)$ where $\gamma =e_C(\frac{P-1}{P}log_C(1)).$ Since $\gamma \in O_L,$  the derivative of $\phi_P(X)-\phi_{P-1} (1)$ is equal to $P,$ and $e_C(K_{P,\infty})$ is $A$-divisible, we conclude that $L/K_P$ is unramified outside $P$ and every place of $K_P$ above $\infty$ is totally split in $L/K_P.$ Furthermore, by lemma \ref{lemma4} :\par
\noindent - if $\phi_{P-1}(1)\equiv 0\pmod{P^2},$ $L/K_P$ is unramified,\par
\noindent - if $\phi_{P-1}(1)\not \equiv 0\pmod{P^2},$ $L/K_P$ is totally ramified at the unique prime of $R_P$ above $P$ (see the proof of lemma \ref{lemma4}).\par
\noindent  Let $H/K_P$ be the Hilbert class field of $R_P,$ i.e. $H/K_P$ is the maximal unramified abelian extension of $K_P$ such that every place in $S_{\infty} (K_P)$ is totally split in $H/K_P.$ Then the Artin symbol induces a $\Delta$-equivariant isomorphism :
$${\rm Pic}(R_P)\simeq {\rm Gal}(H/K_P).$$
Note that $e_{[\omega_P]}G=G,$ where $G={\rm Gal}(L/K_P).$  Thus the Artin symbol induces a $\mathbb F_p[\Delta]$-morphism :
$$\psi :e_{[\omega_P]}(\frac{{\rm Pic}(R_P)}{p{\rm Pic}(R_P)} )\rightarrow {\rm Gal}(L\cap H/K_P).$$
Therefore, by corollary \ref{corollary2} and lemma \ref{lemma6}, we get the following result which explains the congruence of corollary \ref{corollary1}:\par
\begin{theorem}
\label{theorem1}
The morphism  of $\mathbb F_p[\Delta]$-modules induced by the Artin map :
$$\psi :e_{[\omega_P]}(\frac{{\rm Pic}(R_P)}{p{\rm Pic}(R_P)} )\rightarrow {\rm Gal}(L\cap H/K_P),$$
is an isomorphism, where $L=K_P(e_C(\frac{P-1}{P}log_C(1)))$ and $H$ is the Hilbert class field of $R_P.$
\end{theorem}
\subsection{Prime decomposition  of units}${}$\par
\noindent A natural question arise : are there infinitely many primes $P$ such that $\phi_{P-1} (1) \equiv 0\pmod{P^2}\,?$ We end this note by some remarks centered around this question.
\begin{lemma}
\label{lemma7}
Let $N(d)$ be the number of primes $P$ of degree $d$ such that $\phi_{P-1}(1) \not \equiv 0\pmod{P^2}.$ Then :
$$N(d) > \frac{1}{d} (q-1) q^{d-1} - \frac{q}{d(q-1)} q^{d/2}.$$

\end{lemma}
\begin{proof}
Let $N_q(d)$ be the number of primes of degree $d.$ Then :
$$N_q(d) >\frac{1}{d}q^d - \frac{q}{d(q-1)} q^{d/2}.$$
Let $M(d)$ be the number of primes $P$ of degree $d$ such that $\phi_{P-1}(1)\equiv 0\pmod{P^2}.$
Set :
$$V(d)=\sum_{i=0}^{d-1} \frac{L_{d-1}}{L_i}\in A.$$
Then ${\rm deg}_T V(d)= q^{d-1},$ and if $P$ is a prime of degree $d,$ we have by lemma \ref{lemma3}  : $\phi_{P-1}(1)\equiv 0\pmod{P^2}$ if and only if $V(d)\equiv 0\pmod{P}.$
Therefore :
$$M(d)\leq \frac{1}{d}q^{d-1}.$$
\end{proof}
\begin{remark}${}$\par
\noindent  We have :
$$V(2)= 1+T-T^q.$$ Thus $V(2)$ is (up to sign) the product of $q/p$ primes of degree $p.$ Therefore there exist primes $P$  of degree $2$ such that $\phi_{P-1}(1)\equiv 0\pmod{P^2}$ if and only if $p=2,$ and in this case there are exactly $q/2$ such primes.\par${}$\par

\noindent Set $H(X)=\sum_{i=0}^{p-1} \frac{1}{i!} X^i \in \mathbb F_p[X].$ Let $S$ be the set of roots of $H(X)$ in $\bC_{\infty}.$ Note that $\mid S\mid =p-1.$ Let's suppose  that $S\subset \mathbb F_q.$ Let $P$ be a prime of $A$ such that $P$ divides $T^q-T -\alpha $  for some  $\alpha \in \mathbb F_q^*.$ Observe that such a prime is of degree $p.$ Now, for $k=0,\cdots , p-1,$ we have :
$$L_k\equiv \frac{1}{k!} (-\alpha) ^k\pmod{P}.$$
Therefore :
$$V(p)=\sum_{i=0}^{p-1} \frac{L_{p-1}}{L_i}\equiv -\alpha^{p-1} H(\frac{-1}{\alpha})\pmod{P}.$$
Thus there exist at least $(p-1)\frac{q}{p}$ primes $P$ in $A$ of degree $p$ such that $\phi_{P-1}(1)\equiv 0\pmod{P^2}.$
\end{remark}
\begin{lemma}
\label{lemma8}
Let $P$ be a prime of degree $A$ and let $n\geq 1.$ We have an isomorphism of $A$-modules :
$$C(\frac{A}{P^nA})\simeq \frac{A}{P^{n-1}(P-1)A}.$$
\end{lemma}
\begin{proof} We first treat the case $n=1.$ By lemma \ref{lemma3}, we have : $\phi_P(X) \equiv X^{q^d} \pmod{P}.$ Therefore $(P-1) C(A/PA)=\{0\}.$ Now let $Q\in A$ such that $Q.C(A/PA)=\{0\}.$ Then the polynomial $\phi_Q(X) \pmod{P}\, \in (A/PA)[X]$ has  $q^d$ roots in $A/PA.$ Thus ${\rm deg}_TQ\geq d.$This implies that $C(A/PA)$ is a cyclic $A$-module isomorphic to $A/(P-1)A.$\par
\noindent Now let's assume that $n\geq 2.$ By lemma \ref{lemma3}, we have:
$$\forall a\in PA,\, v_P(\phi_P(a))=1+v_P(a).$$
This implies that $C(PA/P^nA)$ is a cyclic $A$-module isomorphic to $A/P^{n-1}A$ and $P$ is a generator of this module. The lemma follows from the fact that we have an exact sequence of $A$-modules :
$$0\rightarrow C(\frac{PA}{P^nA})\rightarrow C(\frac{A}{P^nA})\rightarrow C(\frac{A}{PA})\rightarrow 0.$$
\end{proof}
\noindent We deduce from the above lemma :
\begin{corollary}
\label{corollary3}
Let $P$ be a prime of $A.$ Then $\phi_{P-1}(1)\equiv 0\pmod{P^2}$ if and only if there exists $a\in A\setminus PA$ such that $\phi_a(1) \equiv 0\pmod{P^2}.$
\end{corollary}
\noindent This latter corollary leads us to the following problem:\par
\begin{question} Let $b\in A_+.$ Is it true that there exists a prime $Q$ of $A,$ $Q\equiv 1\pmod{b},$ such that $\phi_Q(1)$ is not squarefree ?
\end{question}
\noindent A positive answer to that question has the following consequence :\par
\begin{lemma}
\label{lemma9} Assume that for every $b\in A_+,$ we have a positive answer to question $1.$ Then, there exist infinitely many primes $P$ such that $\phi_{P-1} \equiv 0\pmod{P^2}.$
\end{lemma}
\begin{proof}
Let $S$ be the set of primes $P$ such that $\phi_{P-1}(1) \equiv 0\pmod{P^2}.$ Let's assume that $S$ is finite. Write $S=\{ P_1, \cdots , P_s\}.$ Set $b=1+\prod_{i=1} ^s (P_i-1)$ (if $S=\emptyset,$ $b=1$). Let $Q$ be a prime of $A$ such that $\phi_Q(1)$ is not squarefree and $Q\equiv 1\pmod{b}.$ Then there exists a prime $P$ of $A$ such that :
$$\phi_Q(1) \equiv 0\pmod{P^2}.$$
Since $\phi_P(1)\equiv 1\pmod{P},$ we have $P\not = Q$ and therefore $Q\in A\setminus PA.$ Furthermore, for $i=1,\cdots ,s,$ $Q$ is prime to $P_i-1.$ Therefore, by lemma \ref{lemma8}, $\phi_Q(1)\not \equiv 0\pmod {P_i^2}.$ Thus $P\not \in S$ which is a contradiction by corollary \ref{corollary3}.
\end{proof}

  \end{document}